\documentclass[11pt]{amsart}
\usepackage[utf8]{inputenc}

\usepackage{amssymb,amsthm,amsmath,amstext}
\usepackage{eucal}

\usepackage{mathrsfs}
\usepackage{bm}
\usepackage[utf8]{inputenc}
\usepackage{comment}
\usepackage[all]{xy}

\usepackage[margin = 2.9cm]{geometry}

\theoremstyle{plain}
\newtheorem{theorem}{Theorem}[section]
\newtheorem*{conjecture*}{Conjecture}
\newtheorem{prop}[theorem]{Proposition}
\newtheorem{lemma}[theorem]{Lemma}
\newtheorem{coro}[theorem]{Corollary}
\theoremstyle{definition}
\newtheorem{remark}[theorem]{Remark}
\newtheorem*{ack}{Acknowledgements}

\def\CC{{\mathbb{C}}}

\def\PP{{\mathbb{P}}}
\def\ZZ{{\mathbb{Z}}}

\def\cJ{{\mathcal{J}}}
\def\cO{{\mathcal{O}}}
\def\cE{{\mathcal{E}}}

\def\cA{{\mathcal{A}}}
\def\cB{{\mathcal{B}}}
\def\cD{{\mathcal{D}}}
\def\cF{{\mathcal{F}}}
\def\cG{{\mathcal{G}}}

\def\cK{{\mathcal{K}}}

\def\cL{{\mathcal{L}}}
\def\cP{{\mathcal{P}}}
\def\cM{{\mathcal{M}}}

\def\cR{{\mathcal{R}}}
\def\cS{{\mathcal{S}}}\def\cU{{\mathcal{U}}}
\def\cC{{\mathcal{C}}}\def\cQ{{\mathcal{Q}}}
\def\cF{{\mathcal{F}}}
\def\cW{{\mathcal{W}}}
\def\ra{{\rightarrow}}
\def\lra{{\longrightarrow}}

\def\fsl{\mathfrak{sl}}

\DeclareMathOperator{\codim}{codim}

\DeclareMathOperator{\Hom}{Hom}
\DeclareMathOperator{\cHom}{\mathcal{H}om}
\DeclareMathOperator{\rank}{rank}
\DeclareMathOperator{\Pic}{Pic}

\DeclareMathOperator{\Sing}{Sing}
\DeclareMathOperator{\GL}{GL}

\DeclareMathOperator{\SU}{SU}
\DeclareMathOperator{\Ext}{Ext}

\usepackage[dvipsnames]{xcolor}
\def\vlad#1{\textcolor{purple}{{\bf Vlad:} #1 {\bf }}}
\def\dan#1{\textcolor{Plum}{{\bf Dan:} #1 {\bf }}}

\def\lau#1{\textcolor{Apricot}{{\bf Lau:} #1 {\bf }}}



\title{Hecke cycles on moduli of vector bundles and orbital degeneracy loci}

\author{V. Benedetti, M. Bolognesi, D. Faenzi, L. Manivel}

\usepackage{hyperref}
\begin{document}

\begin{abstract}
Given a smooth genus two curve $C$, the moduli space $\SU_C(3)$ of rank three semi-stable vector bundles 
on $C$ with trivial determinant is a double cover  in $\PP^8$ branched over a sextic hypersurface, 
whose projective dual is the famous Coble cubic,  the unique cubic hypersurface that is singular along 
the Jacobian of $C$. In this paper we continue our exploration of the connections of such moduli 
spaces with the representation theory of $GL_9$, initiated in \cite{GSW} and pursued in \cite{GS, sam-rains1, sam-rains2, bmt}. Starting from a general trivector $v$ in $\wedge^3\CC^9$, we construct a Fano manifold $D_{Z_{10}}(v)$ 
in $G(3,9)$ as a so-called orbital degeneracy locus, and we prove that it defines a family of Hecke 
lines in $\SU_C(3)$. We deduce that $D_{Z_{10}}(v)$ is isomorphic to the odd moduli space $\SU_C(3, \cO_C(c))$ of 
rank three stable vector bundles on $C$ with fixed effective determinant of degree one. We deduce that 
the intersection of $D_{Z_{10}}(v)$ with a general translate of $G(3,7)$ in $G(3,9)$ is a K3 surface of genus $19$.
\end{abstract}

\maketitle 

\section{Introduction}

The Coble hypersurfaces are among the most fascinating objects in algebraic geometry, 
being connected with an unusual amount of different topics such as abelian varieties and 
their moduli spaces,  moduli spaces of points in $\PP^2$, theta-groups and 
theta-representations \`a la Vinberg, or the arithmetic invariant theory of 
Bhargava-Gross. 
Long after their discovery in the first quarter of 20th century, they were realized to 
be directly connected to certain moduli spaces of vector bundles on curves. Indeed, a Coble
quartic in $\PP^7$ is isomorphic to the moduli space $\SU_C(2)$ of rank two semistable bundles 
with trivial determinant on a genus three curve $C$. For a curve $C$ of genus two, the moduli 
space $\SU_C(3)$ of rank three semistable bundles with trivial determinant is a double 
cover of $\PP^8$, branched over a sextic hypersurface whose projective dual 
is a Coble cubic 
\cite{ortega-coble, nguyen-coble, BB12}.  The geometry of the Coble hypersurfaces reflects in many ways the 
geometry of these curves.

More generally, the moduli spaces $\SU_C(r,L)$ of semi-stable vector bundles of rank $r$ on $C$ 
with fixed determinant $L$ have been thoroughly investigated, from many different perspectives, 
including theta maps, conformal blocks and the Verlinde formula. 
Up to isomorphism, they do not depend on $L$ 
but only on its degree $d$ modulo $r$, and up to sign. Moreover they are smooth when $d$ and 
$r$ are coprime. So the moduli spaces that are connected to the Coble hypersurfaces have each 
one smooth companion, and one can wonder how these Fano manifolds fit into the picture. 

For genus three, we provided an answer to this question in \cite{coblequadric}, where we 
proved that the moduli space $\SU_C(2,L)$ of stable vector bundles with fixed determinant  $L$ of degree one, can be embedded in the 
Grassmannian $G(2,8)$ as the singular locus of a special quadratic section of the Grassmannian. 
This quadratic section is uniquely determined by this property, exactly as the Coble quartic
is uniquely determined by the fact that it is singular along the Kummer threefold of $C$; 
for this reason we coined it the {\it Coble quadric}. Our main purpose in this paper is to
obtain a similar description of the odd moduli space in genus two. 

As in \cite{coblequadric}, our study will be based on the beautiful relationships with 
certain theta-representations, in the sense of Vinberg. To be more specific, 
it was observed 
in \cite{GSW} that a Coble quartic can easily be constructed from a general element in 
$\wedge^4V_8$, if we denote by $V_8$ a complex vector space of dimension eight. Similarly,
a Coble cubic can  be constructed from a general element in $\wedge^3V_9$, if $V_9$
is now a complex vector space of dimension nine. In \cite{bmt} we used the approach of 
{\it orbital degeneracy loci}, whose general theory was developped in \cite{bfmt, bfmt2}, 
to explore this connection further and construct the Coble sextic and its singular locus, 
and even the generalized Kummer fourfold associated to the Jacobian of the curve, which is 
a famous example of hyperK\"ahler manifold. Here we further apply these ideas to linear 
spaces in $\SU_C(3)$. 

Lines in the moduli spaces $\SU_C(r,L)$ have been investigated by many authors, where by 
lines we mean rational curves that have degree one with respect to the ample generator 
of the Picard group. In $\SU_C(3)$ there are three families of lines passing through the 
general point: one family of Hecke lines, and two families of lines contained in the 
$\PP^4$'s of the two natural rulings of the moduli space. Starting from a general element 
$v$ in $\wedge^3V_9$, we will be able to describe these two rulings in terms of orbital
degeneracy loci (Proposition \ref{rulings-SingS}). In order to describe the Hecke lines,
we will first construct an orbital degeneracy locus in $G(3,V_9)$, that we denote 
by $D_{Z_{10}}(v)$, and we will prove that each point of $D_{Z_{10}}(v)$ defines two
Hecke planes in $\SU_C(3)$ (where we define a Hecke plane as a plane whose lines are 
all Hecke lines). It is probably known to the experts that although the family of Hecke
lines is irreducible, the family of Hecke planes has two different components, and this 
is one of the key properties that will allow us to prove our main result (Theorem \ref{thm_main2}):

\medskip\noindent {\bf Theorem.}
{\it For $v$ general in $\wedge^3V_9$, the Fano eightfold $D_{Z_{10}}(v)\subset G(3,V_9)$ 
is isomorphic to 
the odd moduli space $\SU_C(3,\cO_C(c))$, $c\in C$. }

\medskip
An important advantage of this approach by orbital degeneracy loci is that it provides
lots of informations on the embedding of  $\SU_C(3,\cO_C(c))$ into the Grassmannian. For 
example we get a free resolution of its ideal sheaf, from which we can confirm an 
expectation recently expressed by Mukai and Kanemitsu \cite{km}, that the intersection of the 
moduli space with a general translate of $G(3,7)$ in $G(3,9)$ is a K3 surface of genus $19$
(Proposition \ref{k3}). In fact, there is a natural embedding of 
$\SU_C(3,\cO_C(p))$ in $G(3,V_9)$ for any choice of a point $p$ on $C$, and playing with 
these one expects to be able to describe a locally complete family of K3 surfaces of genus $19$. Such locally complete families have been constructed so far for most values of the genus up to $20$, but genus $19$ would be new.

\smallskip
The plan of the paper is the following. In section 2, we recall some useful facts about the 
lines and planes inside the moduli space $\SU_C(3)$, for $C$ a curve of genus $2$. In section 
3 we recall the connections with representation theory and use them to describe explicitely 
the rulings of  $\SU_C(3)$ by $\PP^4$'s, starting from a general element $v$ of $\wedge^3V_9$.
In section 4 we apply the same approach to Hecke lines and planes. We define the orbital 
degeneracy locus $D_{Z_{10}}(v)$, show that it parametrizes pairs of Hecke planes and deduce 
our main Theorem. Finally we describe the minimal resolution of its ideal sheaf and the 
application to K3 surfaces of genus $19$.

\begin{ack}
All authors partially supported by FanoHK ANR-20-CE40-0023.
D.F. and V.B. partially supported by SupToPhAG/EIPHI ANR-17-EURE-0002, Région Bourgogne-Franche-Comté, Feder Bourgogne and Bridges ANR-21-CE40-0017.

We warmly thank Christian Pauly for useful discussions, and Shigeru Mukai and Akihiro Kanemitsu for sharing the results of \cite{km}. 
\end{ack}

\section{Moduli space of rank $3$ vector bundles on genus $2$ curves}

\subsection{Moduli spaces}
We will use the same notations as in \cite{coblequadric}: we denote by 
$\mathrm{U}_C(r,d)$ the moduli
space of semi-stable vector bundles on a curve $C$ of rank $r$ and determinant of degree $d$. If $L$ is a degree $d$ line bundle on $C$, we will denote by $\SU_C(r,L)$ the subvariety of $U_C(r,d)$ parametrizing vector bundles of fixed determinant $L$; moreover $\SU_C(r):=\SU_C(r,\cO_C)$. Since all the moduli spaces $\SU(r,L)$ are (non canonically) isomorphic when the degree of $L$ is fixed, we will also denote their isomorphism class by $\SU_C(r,d)$; it does depend on $d$ only modulo $r$. 

\subsection{Theta maps} 
In this paper we focus on the case where $C$ has genus two and the rank $r=3$. 
The moduli space $\SU_C(3)$ has Picard rank one, and the positive generator of the Picard group defines the Theta map 
$$\theta: \SU_3(C)\longrightarrow |3\Theta|=\PP(V_9)$$
to the linear system $|3\Theta|$ on $C$. This map is a double cover branched over 
a sextic hypersurface $\cS$, whose projective dual is the Coble cubic, characterized as
the unique cubic hypersurface in $  \PP(V_9^\vee)$ that is singular along the Jacobian 
of $C$ \cite{ortega-coble, nguyen-coble, beauville-coble}.

The sextic $\cS$ is also singular, in fact its singular locus identifies with that 
of $\SU_C(3)$, which has dimension five and can be described set theoretically as 
the locus of properly semi-stable vector bundles, equivalent to a direct sum 
$E=F\oplus \det(F)^\vee$ for some rank two vector bundle $F$ of degree zero 
on $C$. As a deeper stratum, we 
get the fourfold $\cK$ parametrizing vector bundles 
$E=L_1\oplus L_2\oplus L_3$, for three degree zero line bundles $L_1, L_2, L_3$ with trivial 
tensor product. When two  among the three bundles $L_1,L_2,L_3$ coincide, we get a surface that is a singular model of the Jacobian of $C$. Its singular locus, where the three bundles $L_1,L_2,L_3$ all coincide, is in bijection with the $81$ three-torsion points of the Jacobian.

\subsection{Lines in the moduli space} 
Rational curves in the moduli spaces $\SU_C(r,d)$ were extensively studied, see e.g. 
\cite{nr75, NarRamHecke,  sun, mok-sun, pal, mustopa-teixidor}.
Restricting to $g=2$,  $r=3$ and $d=0$, the results of 
\cite{mustopa-teixidor} show that there exist three different families of covering 
lines, i.e. families of rational curves of degree one with respect to the Theta embedding
and passing through a general point of the moduli space. We will 
denote these three families, as subvarieties of the Hilbert scheme,  
by $\cF_H$, $\cF_{R_1}$,  $\cF_{R_2}$, and we will denote by $\overline{\cF}_H$, 
$\overline{\cF}_{R_1}$,  $\overline{\cF}_{R_2}$ their bijective images inside the
Grassmannian $G(2,V_9)$. They all have dimension eleven, which is the expected dimension.


\subsubsection{Hecke lines}
The family $\cF_H$ of Hecke lines is most naturally defined as parametrizing lines 
in planes that are constructed as follows. Start with a point $p\in C$ and a 
stable rank three vector bundle $E$ on $C$ with $\det(E)=\cO_C(p)$. Any non zero
linear form on the fiber $E_p$ defines a rank three vector bundle $F$ on $C$ as the 
kernel of the composition $E\ra E_p\ra \cO_p$. 
Since  $E$
is stable of degree one, by \cite[Remark 5.2 (v)]{NarRamHecke} it is 
$(0,1)$-semistable, and then by \cite[Lemma 5.5]{NarRamHecke} the vector bundle 
$F$ is automatically semistable.

We thus get a linear embedding of $\PP(E_p^\vee)$ inside $\SU_C(3)$, whose image we will 
denote by $\Pi(E,p)$ (this is a  {\it good Hecke cycle} in the terminology of 
\cite{NarRamHecke}; the fact that the embedding is linear is e.g. in \cite{mustopa-teixidor}). 
By the first paragraph of the proof of \cite[Theorem 5.13]{NarRamHecke}, the resulting 
morphism 
$$\phi_p : \SU_C(3,\cO_C(p)) \lra G(3,V_9)$$
is injective.


Note also that the 
point $p$ can be uniquely recovered from any good Hecke plane. Indeed, the image of $\PP(E_p^\vee)$ in $\SU_C(3)$ is defined by the rank three bundle $M$ over $C\times 
\PP(E_p^\vee)$ fitting into the exact sequence 
$$0\lra M\lra p_1^*E\lra p_1^*\cO_p\otimes p_2^*\cO_{\PP(E_p^\vee)}(1)\lra 0$$
\cite[page 362]{NarRamHecke}.
For a point $q\ne p$ of $C$, the restriction $M(q)$ of $M$ to 
$\{q\}\times \PP(E_p^\vee)$ 
is a trivial bundle on $\PP^2$, while $M(p)$ is the direct sum of the hyperplane bundle and 
the dual tautological bundle, in particular it is non trivial. 
This allows to distinguish $p$ among the points of $C$. 

\smallskip 
The lines in the planes $\Pi(E,p)$ are  the Hecke lines. They are 
specified by fixing a one-dimensional subpace $L$ of $E_p$, and considering the planes  $P\subset E_p$ that contain $L$; we will denote the 
corresponding Hecke line by $\delta(E,p,L)$. In fact the previous argument works 
over such a line:

\begin{lemma}\label{unique}
A Hecke line $\delta(E,p,L)$ uniquely determines the vector bundle $E$ and also
the point $p$ in general.
\end{lemma}

As a result, the projectivization $\hat{\cF}_H$ of the dual universal bundle on
what is denoted 
$U_C(3,C)$ in \cite{NarRamHecke} (the moduli space of stable bundles with effective 
determinant of degree one) maps birationally to $\cF_H$, 
and fibers over the curve $C$.

\subsubsection{Lines in the rulings}
Another way to construct lines in $\SU_C(3)$ is to use extensions between bundles of
smaller ranks. Consider stable vector bundles $E_1, E_2$ with $\rank(E_1)=r_1$ and $\rank(E_2)=r_2=3-r_1$, $r_1$ being either $1$ or $2$. Suppose that $E_1$ has degree $-1$ 
and $E_2$ degree $1$, with $\det(E_1)=\det(E_2)^{-1}$. So if $r_1=1$, the line bundle 
$E_1$ is determined by $E_2$, and vice versa: if $r_1=2$, the line bundle $E_2$ is 
determined by $E_1$. The space of extensions 
between $E_1$ and $E_2$ is five dimensional, and we get  a linear embedding of $\PP(\Ext^1(E_2,E_1))\simeq\PP^4$ inside $\SU_C(3)$. Hence we obtain two rulings $R_1$ and $R_2$ in
four-dimensional projective spaces. The families $\cF_{R_1}$ and  $\cF_{R_2}$ consist
in the lines contained in these two rulings.

\begin{lemma}\label{unique-ruling}
A line in $\cF_{R_1}$ (or $\cF_{R_2}$) that is  contained in the stable locus of 
the moduli space is contained in a unique linear
space of the ruling.
\end{lemma} 

\proof Consider a line $\delta$ in $\SU_C(3)$ which is contained in two different members
of the ruling $\cR_1$, corresponding to pairs of vector bundles $(E_1,E_2)$ and $(E'_1,E'_2)$, and two planes $A\subset \Ext^1(E_2,E_1)$, $A'\subset \Ext^1(E'_2,E'_1)$. Each point of $\PP(A)$ and $\PP(A')$ can be identified to a point  of $\delta=\PP(B)$ , and
we can choose isomorphisms $A\simeq B\simeq A'$ compatible with these identifications. 
Hence an element $e$ of  $\Hom(B,\Ext^1(E_2,E_1))\simeq \Ext^1(p_1^*E_2, p_1^*E_1\otimes p_2^*\cO_{\PP(B)}(1))$, where the latter Ext group is taken between vector bundles on 
$C\times \PP(B)=C\times\delta$, with $p_1$ and $p_2$ the two projections. This defines 
on this surface an extension (see \cite[Proposition 2.5]{mustopa-teixidor})
$$0\lra p_1^*E_1\otimes p_2^*\cO_{\PP(B)}(1)\lra \cE\lra p_1^*E_2\lra 0.$$
Similarly, we get an element $e'$ of  $\Hom(B,\Ext^1(E'_2,E'_1))\simeq \Ext^1(p_1^*E'_2, p_1^*E'_1\otimes p_2^*\cO_{\PP(B)}(1))$, and an extension
$$0\lra p_1^*E'_1\otimes p_2^*\cO_{\PP(B)}(1)\lra \cE'\lra p_1^*E'_2\lra 0.$$
We claim that  $\cE\simeq\cE'$. Since $\delta$ is supposed not to meet the 
locus of properly semistable bundles (the singular locus of the moduli space), the isomorphism class of $\cE_{C\times \{t\}}$ is uniquely determined at every point $t\in\delta$, and has to coincide with that of  $\cE'_{C\times \{t\}}$. Since this is true for any $t\in\delta$, we deduce from  \cite[Lemma 2.5]{ramanan}
that $\cE'=\cE\otimes p_2^*N$ for some line bundle $N$ 
on $\delta$; computing determinants we see that $N$ is in fact trivial. 

But then the composition 
$$p_1^*E_1\otimes p_2^*\cO_{\PP(B)}(1)\lra\cE\simeq\cE'\lra p_1^*E'_2$$
clearly has to vanish, hence there is an induced morphism from $p_1^*E_1\otimes p_2^*\cO_{\PP(B)}(1)$
to $p_1^*E'_1\otimes p_2^*\cO_{\PP(B)}(1)$, and then from $E_1$ to $E'_1$. Exchanging 
$e$ and $e'$ we deduce that in fact $E_1\simeq E'_1$ and  $E_2\simeq E'_2$. \qed

\medskip
In the construction of the rulings, the rank two bundle (or its dual) is 
parametrized by $U_C(2,1)$ and determines the rank one bundle. So the rulings are parametrized by $U_C(2,1)$. By the previous Lemma, the corresponding families of lines $\cF_{R_1}$ and  $\cF_{R_2}$ are the birational images of $\hat{\cF}_{R_1}$ and  $\hat{\cF}_{R_2}$,
which are $G(2,5)$-fibrations over $U_C(2,1)$. 
Notice that since $U_C(2,1)$ is a fiber bundle 
over the abelian surface $\Pic^{1}(C)$, this is also the case for $\hat{\cF}_{R_1}$ and  $\hat{\cF}_{R_2}$.

\subsubsection{Action of  the involution} \label{remark_two_rulings}
According to \cite{ortega-coble}, the covering involution $\tau$ of the double cover $\theta$ of $\SU_C(3)$  is given by $E\mapsto \tau(E):=\iota^* E^\vee$,  where 
$\iota$ is the hyperelliptic involution of $C$.

\begin{lemma}
The two rulings $R_1$ and $R_2$ are exchanged by $\tau$. 
\end{lemma}

\begin{proof} 
If the bundle $E$ fits into an exact sequence $$0\to E_1 \to E \to E_2 \to 0,$$
with $\rank(E_1)=2$, $\deg(E_1)=-1$, $\rank(E_2)=1$, $\deg(E_2)=1$, then $\iota^* E^\vee$ fits into the exact sequence
$$0\to \iota^* E_2^\vee \to \iota^* E^\vee \to \iota^* E_1^\vee \to 0 ,$$
with $\rank(\iota^* E_2^\vee)=1$, $\deg(\iota^* E_2^\vee)=-1$, $\rank(\iota^* E_1^\vee)=2$, $\deg(\iota^* E_1^\vee)=1$. This implies the claim.
\end{proof}

The upshot is that the images 
of $\cF_{R_1}$ and $\cF_{R_2}$ inside $G(2,V_9)$ form a single family   $\overline{\cF}_{R}:=\overline{\cF}_{R_1}=\overline{\cF}_{R_2}$, and their fibrations
over $\Pic^1(C)$ descend to $\overline{\cF}_{R}$.

\begin{lemma}\label{linespres}
The family $\cF_H$ of Hecke lines  is preserved by $\tau$. 
\end{lemma}

\begin{proof}
This is clear from the previous Lemma since the lines are preserved by $\tau$.
But let us be more precise. Let $E$ be a stable rank three vector bundle such that $\det(E)=\cO_C(p)$ for a point $p\in C$. 
 Consider  a complete flag $L\subset P\subset E_p$. There is a commutative diagram 

\begin{equation*}\label{heckediagram}
\xymatrix@-2ex{
& 0 & & & \\ 
& \cO_p\otimes P/L\ar[u] & & 0 & \\
 0\ar[r]&F_P\ar[r]\ar[u] & E \ar[r] & \cO_p\otimes E_p/P\ar[r]\ar[u]& 0 \\
 0\ar[r]&G_L^\vee\ar[r]^{\gamma_L}\ar[u] & E \ar[r]\ar@{=}[u] & \cO_p\otimes E_p/L\ar[u]\ar[r]& 0 \\
& 0\ar[u]& & \cO_p\otimes P/L\ar[u]& \\
& & & 0\ar[u]& 
}
\end{equation*}

The bottom horizontal exact sequence defines a vector bundle $G_L^\vee$, whose dual $G_L$ 
has again rank three and determinant $\cO_C(p)$. Moreover $G_L$ is also stable. Indeed, $F_P$ is (semi)stable of degree $0$, hence  $(0,1)$-(semi)stable.
By \cite[Lemma 5.5]{NarRamHecke} $G_L^\vee$ is then $(0,0)$-(semi)stable, hence 
(semi)stable as well as its dual.

The plane $\Pi(E,p)$ in $\SU_C(3)$ parametrized by $\PP(E_p^\vee)$ consists in the vector bundles $F_P$,
for $P\subset E_p$. We get a line in this plane, that is, a Hecke line, 
by imposing to the 
planes $P$ to contain a common line $L$. Then dualizing the first vertical exact sequence in the 
previous diagram, and pulling-back by the hyperelliptic involution, we get 
$$0\lra \iota^*F_P^\vee \lra \iota^*G_L \lra \cO_{\iota(p)} \lra 0,$$
confirming that the image of our Hecke line is again a Hecke line. 
\end{proof}

Nevertheless, observe that the bundle $\iota^*G_L$ that defines this Hecke line does really depend on $L$,
not only on $E$. This will be a direct consequence of the next 

\begin{lemma}
The map from $\PP(E_p)$ to $\SU_C(3,\cO_C(p))$, sending $L$ to $G_L$, is non constant.
\end{lemma}

\begin{proof}
The proof mimicks the proof of \cite[Lemma 5.9]{NarRamHecke}, where it is proved that the differential of the morphism $L\mapsto G_L$ is injective in a context 
that needs a little adaptation, as follows. The differential is computed in \cite[Lemma 5.10]{NarRamHecke} in a rather general setting. 

To deduce it is injective in our situation as well, 
let us onsider the exact sequence 
$$0\to G_L^\vee \to E \to \cO_p\otimes E_p/L\to 0.$$ 
By letting $L$ vary in $\PP(E_p)$, we get a family of extensions 
$$0\to \cG^\vee \to p_1^*\cE \to p_2^*\cQ \otimes \cO_{\{p\}\times \PP(E_p)}\to 0$$ 
on $C\times \PP(E_p)$, where $\cQ$ is the rank two tautological bundle on $\PP(E_p)$. By \cite[Lemma 5.10]{NarRamHecke},   the differential of $L\mapsto G_L$ is given by the connecting homomorphism 
$$\cHom((\cG^\vee|_{\{p\}\times \PP(E_p)},\cQ) \to R^1p_{2*} \mathcal{E}nd(\cG).
$$
But this fits in the exact sequence
$$ 0\to p_{2*} \mathcal{E}nd(\cG) \to p_{2*}\cHom(\cG^\vee, p_1^* \cE) \to \cHom((\cG^\vee|_{\{p\}\times \PP(E_p)},\cQ) \to R^1 p_{2*} \mathcal{E}nd(\cG)  .$$ Since $\cG^\vee$ is a family of stable bundles, $p_{2*} \mathcal{E}nd(\cG)=\cO_{\PP(E_p)}$. Similarly to what happens in the proof of \cite[Lemma 5.9]{NarRamHecke}, the result will follow if we can show that $\Hom(G_L^\vee,E)$ is one dimensional, which is essentially a slight modification of \cite[Lemma 5.6]{NarRamHecke}.

In order to check this, the main observation is the following. The morphism 
$\gamma_L  : G_L^\vee\ra E$ in the diagram above has full rank outside $p$. 
Suppose $\gamma  : G_L^\vee\ra E$ is another a morphism , generically of full rank. 
Then $\det(\gamma)$
defines a section of $\det(G_L)\otimes \det(E)=\cO_C(2p)$. But this line bundle has a unique global section, vanishing only (with multiplicity two) at $p$. So $\gamma$ also has to
be an isomorphism outside $p$. But if the dimension of $\Hom(G_L^\vee,E)$ was bigger
than one, we could construct morphisms $s\gamma_L+t\gamma$ whose rank would drop at any chosen point of $C$, and this would yield a contradiction. 
\end{proof}

\begin{prop}
The image of the plane $\Pi(E,p)$ by 
the involution $\tau$ is 
$$\tau (\Pi(E,p))=\Pi^*(\iota^*E,\iota(p)),$$ 
where $\Pi^*(F,q)$ is the plane that parametrizes the duals of the bundles 
parametrized by $\Pi(F,q)$. In particular, it is {\it not} a plane of the same type.
\end{prop}

\begin{proof}
Indeed, if we had $\tau (\Pi(E,p))=\Pi(E',p')$ for some 
bundle $E'$ and some point $p'$, then by the proof of Lemma \ref{linespres} we would deduce that $G_L=\iota^*E'$ does  not depend on $L$. 
\end{proof}
 

Let us define a {\it Hecke plane} in $\SU_C(3)$ as a plane whose lines are all Hecke lines. We have defined two distinct families of Hecke planes parametrized by $U_C(3,C)$
(which is nine dimensional), and we want prove that there is no other Hecke plane through the general point of $\SU_C(3)$. The key point to prove this claim is the following result \cite[Theorem 3]{hwang-hecke}:

\begin{prop}
The VMRT at a general point $[F]$ of the family of Hecke lines is the image of the relative flag 
manifold $Fl(F)\stackrel{\psi}{\lra}C$ by the linear system $|\cO_{Fl(F)}(1,1)\otimes 
\psi^*K_C|$, into the projectivized tangent space of $\SU_C(3)$ at $F$, that is, $\PP (H^1(C,\mathcal{E}nd_0(F)))\simeq \PP (H^0(C,S_{2,1}(F)\otimes K_C)^\vee).$
\end{prop}

\begin{coro}
A general Hecke line is contained in exactly two Hecke planes.
\end{coro}

\proof Consider a general Hecke line $\delta(E,p,L)$ in $\SU_C(3)$, and let $\Pi$ be a Hecke plane 
containing this line. Let $[F]$ be a general point on $\delta(E,p,L)$, and let us apply the previous 
Proposition. In the VMRT at $[F]$, $\delta(E,p,L)$ defines a point $\delta$, and  $\Pi$ defines a projective line $\pi$ through $\delta$. Observe that any line in $\Pi$ is a Hecke line, hence 
of the form $\delta(E',p',L')$; since the point $p'$ is uniquely determined in general, we get a rational map from the dual plane of lines in $\Pi$, to $C$, and such a map must be constant, equal to $p$. 
This means that $\pi$ is contained in the fiber over $p$ of the relative flag bundle $Fl(F)$. 
So we just need to consider the ordinary three-dimensional flag manifold $Fl(1,2,\CC^3)$. And then it is clear that any point $(L_0\subset P_0)$ is contained in
exactly two lines, parametrizing flags of the form  $(L_0\subset P)$ or
$(L\subset P_0)$. \qed


\medskip
It would be interesting to describe the VMRT at a general point of 
the families $\cF_{R_1}$ and $\cF_{R_2}$. They must be $\PP^3$-bundles over the families of $\PP^4$'s in the rulings passing through the general point  $[E]\in SU_C(3)$, which are parametrized by curves $C_1$ and $C_2$. Note that like the rulings $R_1$ and $R_2$, the curves 
$C_1$ and $C_2$ map to $\Pic^1(C)$, and in fact these curves have a simple description in terms of theta divisors. Each of them parametrizes line bundles $L$ of degree -1 such that $h^0(L^\vee\otimes E)\neq 0$. The locus of line bundles $L^\vee \in Pic^1(C)$ that verify this condition is just the \it theta divisor \rm associated with $E$, which in this case is a divisor in $\Pic^1(C)$ belonging to the linear system $|3\Theta|$. For general $E$, the associated theta divisor is smooth and irreducible. An easy computation shows that the generic element of $|3\Theta|$ is a curve of genus $10$.

 \medskip
Let us summarize the results we will use in the sequel. 

\begin{prop} \label{recap}
The moduli space $\SU_C(3)$ admits three covering families of lines:
\begin{itemize}
\item The family $\cF_H$ of Hecke lines, which is birationally fibered over $C$.
\item Two families of lines $\cF_{R_1}$ and  $\cF_{R_2}$ contained in the rulings, 
which are 
birationally fibered over $\Pic^1(C)$ and exchanged by the involution $\tau$; they both
descend to the same family $\overline{\cF}_R$ inside $G(2,V_9)$, which is still 
birationally fibered over the abelian surface $\Pic^1(C)$.
\end{itemize}
The moduli space $\SU_C(3)$ also admits  four covering families of planes:
\begin{itemize}
\item Two families $\cP_H$, $\cP_{H^*}$ of Hecke planes, both
birationally fibered over $C$.
\item Two families  $\cP_{R_1}$ and  $\cP_{R_2}$ of planes contained in the rulings,
both birationally fibered over $\Pic^1(C)$.
\end{itemize}
\end{prop}

\proof The only point that does not follow from the previous discussion is that there is not other plane $\Pi$ through the general point of $SU_C(3)$, which is  made of lines contained in one of the rulings. What we have 
to exclude is that the lines in $\Pi$ belong to $\PP^4$'s, say from
$R_1$, that move with the line. In this situation, we would get a map 
from the dual plane $\Pi^\vee$ to $R_1$, that we could compose with 
the projection to $\Pic^1(C)$. The composition would need to be constant, which means that the line bundle we denoted $E_1$ must be constant. 

But then, consider the pencil of lines in $\Pi$ passing through a fixed general point $[E]$ of $\SU_C(3)$. For each of these lines we get a morphism $E_1\lra E$, and the quotient bundle has to change with the line. So we deduce that $F=\Hom(E_1,E)$ has at least two global sections. 
This means that $F$ belongs to the Brill-Noether stratum $\cW_{3,3}^1$,
which according to \cite{bgn} has the expected dimension given by the Brill-Noether number $\rho_{3,3}^1=6$. Since $F$ determines $E_1$ 
up to $3$-torsion, this leaves only six parameters also for $E$, 
which cannot be general. 
\qed

\section{ODL interpretations}

In this section we summarize the results of \cite{GSW} and \cite{bmt} about the 
relationships between the moduli space $\SU_C(3)$, the Coble cubic and a skew-symmetric
cubic tensor in nine variables.

\subsection{On the Coble side}
The starting point is a general $v\in \wedge^3V_9$. By the Borel-Weil theorem, 
there is a natural identification 
$$\wedge^3V_9\simeq H^0(\PP(V_9^\vee),\wedge^2\cU(1)),$$
where $\cU$ denotes the rank eight tautological bundle. This allows to consider 
the skew-symmetric cubic tensor $v$ as a family of skew-symmetric bilinear forms, 
and then the loci where these forms drop rank: first a Pfaffian cubic hypersurface,
and then its singular locus where the rank drops to six. This turns out to be an
abelian surface, and the Pfaffian hypersurface is the corresponding Coble cubic, being singular exactly on the abelian surface. 

\subsection{On the dual side}
The Borel-Weil theorem gives another identification 
$$\wedge^3V_9\simeq H^0(\PP(V_9),\wedge^3\cQ),$$
where $\cQ$ denotes the rank eight tautological quotient bundle. 
This allows to consider 
the skew-symmetric cubic tensor $v$ in nine variables as a family of cubic tensors in eight variables,
which are considerably simpler. In particular, $\wedge^3V_8$ admits finitely many $\GL(V_8)$-orbits, and for each orbit closure $Y\subset \wedge^3V_8$ there is an orbital degeneracy locus (ODL) $D_Y(v)\subset \PP(V_9)$. Set theoretically, this is simply the set 
of points $x$ for which $v$ defines an element of  $\wedge^3\cQ_x\simeq \wedge^3V_8$
that falls into $Y$. For $v$ generic, we need only consider the orbits $Y_i$ of codimension $i\le 8$, since this will also be the codimension of  $D_{Y_i}(v)\subset \PP(V_9)$. 
Moreover, this locus will be almost always singular, with singular locus  $D_{\Sing(Y_i)}(v)$,
and it will come with some natural desingularisations, induced by desingularisations 
of $Y$ of a special type, called Kempf collapsings. Concretely, in this case we get the 
following loci:
\begin{itemize}
    \item The Coble sextic 
    $$\cS:=\{[U_1]\in \PP(V_9)\mid \exists U_6\supset U_1, \;\; v\in U_1\wedge (\wedge^2V_9)
+    \wedge^2 U_6 \wedge V_9\}\subset \PP(V_9).$$
    \item Its singular locus (five dimensional, corresponding to strictly semistable bundles)
    $$ \Sing(\cS):= \{[U_1]\in \PP(V_9)\mid \exists U_5\supset U_2 \supset U_1, \;\;
    v\in  U_2\wedge (\wedge^2V_9) + \wedge^2 U_5 \wedge V_9 \}. $$
    \item The fourfold $\cK$ parametrizing rank three vector bundles with trivial determinant, splitting as a sum of three line bundles: 
    $$ \cK:= \{[U_1]\in \PP(V_9)\mid \exists U_6\supset U_3\supset U_1, \;\, 
    v\in  U_1\wedge (\wedge^2V_9) + U_3\wedge U_6 \wedge V_9+\wedge^3 U_6 \} .$$
    As proved in \cite{bmt}, $\cK$ is a singular model of the generalized Kummer fourfold associated to the Jacobian of $C$. Moreover, its singular locus (when two of the three line bundles are equal) is a surface which is birational to the Jacobian itself, but with $81$ singular points in bijection with the  $3$-torsion points (when the three line bundles are equal).
\end{itemize}

\subsection{The rulings}
We have seen that the two rulings of $\SU_C(3)$ by $\PP^4$'s are parametrized by 
pairs of bundles $(E_1,E_2)$, of ranks $(1,2)$ or $(2,1)$, such that $\det(E_1)\otimes\det(E_2)=\cO_C$. Up to a shift in the degrees, this is very 
similar to the data that defines a strictly semistable rank three vector bundle on 
$C$, hence a point
in $\Sing(\cS)$. So we expect the latter to be the parameter space for some natural
families of five-dimensional subspaces of $V_9$. This is exactly the content of the
next two statements. 

Since $\SU_C(3)$ is a double cover of $\PP(V_9)$ branched over
the sextic $\cS$, a $\PP^4=\PP(U_5)$ in $\PP(V_9)$ will lift to a pair of $\PP^4$'s
in $\SU_C(3)$, meeting in codimension one, exactly when the sextic hypersurface 
restricts on $\PP(U_5)$ to a double cubic. We will check this phenomenon does happen 
by rather indirect arguments. Define 
 $$ \overline{\Sing}(\cS):= \{(U_1\subset U_2\subset U_5)\in Fl(1,2,5, V_9),  \;\;
    v\in  U_2\wedge (\wedge^2V_9) + \wedge^2 U_5 \wedge V_9 \}. $$

\begin{lemma}
The projection $ \overline{\Sing}(\cS)\ra \Sing(\cS)$ is generically finite of degree two. 
\end{lemma}

\begin{proof}
This is contained in \cite[Remark 3.1]{bmt}.
\end{proof}

\begin{prop}\label{rulings-SingS}
Let  $(U_1\subset U_2\subset U_5)$ be general in $ \overline{\Sing}(\cS)$. 
Then $\PP(U_5)\cap \cS$ is a double cubic in $\PP(U_5)\simeq\PP^4$, and
$\theta^{-1}(\PP(U_5))$ is the union of two copies of $\PP^4$ in $\SU_C(3)$.
\end{prop}

\begin{proof}
Let  us fix a general $(U_1\subset U_2\subset U_5)$ in $ \overline{\Sing}(\cS)$. 
This means that $$v\in  U_2\wedge (\wedge^2V_9) + \wedge^2 U_5 \wedge V_9.$$
The intersection $\PP(U_5)\cap \cS$ parametrizes the flags 
$(W_1\subset W_6)$, with $W_1\subset U_5$, such that 
$$v\in W_1\wedge (\wedge^2V_9)+  \wedge^2 W_6 \wedge V_9.$$
We need to understand the relative 
position of the flags $(U_1\subset U_2\subset U_5)$ and $(W_1\subset W_6)$. 

\begin{lemma} 
$\dim(W_6\cap U_2)=2$ and $\dim(W_6\cap U_5)=4$.
\end{lemma}

\noindent {\it Proof of the Lemma.} We will argue as follows. Let 
$$\cA= U_2\wedge (\wedge^2 V_9) + \wedge^2 U_5 \wedge V_9, \qquad \cB =W_1\wedge \wedge^2 V_9 + \wedge^2 W_6 \wedge V_9.$$
The dimension of $\cA$ is $62$. 
We can stratify the space 
of flags $(W_1\subset W_6)$, with $W_1\subset U_5$, by the possible values of 
$\dim(W_1\cap U_2)$,  
$\dim(W_6\cap U_2)$ and $\dim(W_6\cap U_5)$. On each stratum $S$, the quotient 
 $\cP:=\cA/ \cA\cap \cB$
has a fixed dimension, hence defines a vector bundle of a certain rank. Moreover 
this bundle is generated by global sections and $v$ defines such a global section.
Although $S$ is in general not compact, we can apply the general Bertini type arguments 
to the space $\Gamma=U_2\wedge (\wedge^2 V_9) + \wedge^2 U_5 \wedge V_9$ of global sections,
which is finite dimensional and generates $\cP$ everywhere on $S$. 

Remark that we consider the same space $U_2\wedge (\wedge^2 V_9) + \wedge^2 U_5 \wedge V_9$ in two different ways: first as a trivial bundle $\cA$ on $S$, then as a space of global sections of the quotient bundle $\cP$.
We will be able to 
conclude that when $\rank_S(\cP)>\dim(S)$, the general $v\in U_2\wedge( \wedge^2 V_9) + \wedge^2 U_5 \wedge V_9$ defines a section that vanishes nowhere, which means that there
exists no flag  $(W_1\subset W_6)$ satisfying the corresponding conditions. If $\rank_S(\cP)\le \dim(S)$, the general section either vanishes nowhere, or on a locus 
of dimension $\dim(S)-\rank_S(\cP)$, and we need this difference to be equal to 
three (at least) since we are looking for a hypersurface in $\PP^4$. Note in particular that we can, and will always suppose in the sequel that $\dim(W_1\cap U_2)=0$, 
since  $\PP(U_2)$ is only one dimensional. 

When convenient, 
we will occasionally replace $\cP$ by some quotient $\overline{\cP}$, 
to which the same arguments will a fortiori apply. In some cases we will even need that this quotient is in fact obtained by pull-back from another parameter space $\overline{\cS}$; under the condition that $\rank(\overline{\cP})>\dim\overline{\cS}-3$, this will allow as well to eliminate the stratum $\cS$.  
In the end, a unique stratum will survive and we will have proved the Lemma. 
More precisely, we will prove that for all but one couple of values $(\dim(W_6\cap U_2),\dim(W_6\cap U_5))$ the section defined by $v$ vanishes on 
a locus of dimension at most two. So there will be a unique couple of values $(\dim(W_6\cap U_2),\dim(W_6\cap U_5))$ describing the intersection dimensions of a general flag $W_1\subset W_6$ inside $\PP(U_5)\cap \cS$. The conclusion will hold 
for any flag $(U_2\subset U_5)$ and for the general $v\in U_2\wedge (\wedge^2 V_9) + \wedge^2 U_5 \wedge V_9$, hence conversely for the general $v$ and the general flag $(U_2\subset U_5)$ such that $v\in U_2\wedge (\wedge^2 V_9) + \wedge^2 U_5 \wedge V_9$. 

Let us discuss the different possibilities for $(\dim(W_6\cap U_2),\dim(W_6\cap U_5))$, which can a priori be $(0,2), (1,2), (2,2), (0,3), (1,3), (2,3), (1,4), (2,4), (2,5)$. 

\begin{itemize}
 \item[(0,2)] This is the open stratum.  We compute that 
 $$\cA\cap\cB=W_1\wedge U_5\wedge V_9+U_2\wedge (\wedge^2W_6)+
 (W_6\cap U_5)\wedge W_6\wedge U_5$$
 has dimension $46$. Since this is the first time we make such a computation, let us explain once and for all how we proceed; the bravest readers should be able to check the next computations by themselves. Let $v_1,\ldots,v_9  $ be a basis of $V_9$ such that $U_2=\langle v_1,v_2 \rangle$ and $U_5=\langle v_1,\ldots,v_5\rangle$. Since we are on the stratum indexed by the intersection dimensions $(0,2)$, we can suppose that $W_1=\langle v_3 \rangle$ and $W_6=\langle v_3,v_4,v_6,v_7,v_8,v_9 \rangle$. Then a brute force computation shows that a basis of $\cP:=\cA/\cB$ is given by the classes, if 
 we denote $v_{ijk}:=v_i\wedge v_j \wedge v_k$,  of the following $16$ trivectors: $v_{145}$, $v_{156}$, $v_{157}$, $v_{158}$, $v_{159}$, $v_{245}$, $v_{256}$, $v_{257}$, $v_{258}$, $v_{259}$, $v_{124}$, $v_{125}$, $v_{126}$, $v_{127}$, $v_{128}$, $v_{129}$.
 So $\cP$ has rank $16$ and $\cA\cap \cB$ has dimension $62-16=46$, as claimed. 
 
 Since the open stratum has dimension $19=16+3$, at first sight there is no obstruction. But let  $W_2:=W_6\cap U_5$, $W_4:=W_2+U_2$ and $W_8:=W_6+U_2$. The parameter space for such flags  $(W_1\subset W_2\subset W_4\subset W_8)$, with $U_2\subset W_4\subset U_5$ has dimension $11$ (six dimensions 
 for the choice of $W_2$ in $U_5$, transverse to $U_2$; then one extra dimension for the choice of $W_1$; $W_4$ is uniquely determined; finally four extra dimensions for the choice of $W_8$ containing $W_4$). Let 
 $$\cC:= W_2\wedge U_5\wedge U_9+U_2\wedge (\wedge^2W_8).$$
 We check that $\cA\cap \cB\subset \cC\subset\cA$ and that $\cC$ has dimension $53$, hence codimension $9$ in $\cA$.
 Therefore the zero locus of the section of the bundle $\cA/\cC$ induced by our 
 general $v$ must have dimension at most $11-9=2<3$. This means that a general point in the intersection $\PP(U_5)\cap \cS$ does not belong to this stratum.
 \item[(1,2)] Let $T_1:= W_6\cap U_2$ and $W_2:=W_6\cap U_5$. The parameter space for such flags has dimension 
 $17$: one dimension for the choice of $T_1$ in $U_2$, four dimensions for the choice of $W_1$ in $U_5$, then twelve dimensions for the choice of $W_6$ containing $W_2=T_1+W_1$. We compute that    
 $$\cA\cap\cB=W_1\wedge U_5\wedge V_9+U_2\wedge (\wedge^2W_6)+T_1\wedge W_6\wedge V_9$$
 has dimension $46$, so the rank of $\cP$ is $16=\dim(\cS)-1$, and we are safe. 
 \item[(2,2)] In this case $W_6\cap U_5=W_6\cap U_2=U_2$. But this implies that $W_1\subset U_2$, 
 which we excluded right from the beginning. 
 \item[(0,3)] Let $W_3:=W_6\cap U_5$. 
 The parameter space for such flags  $(W_1\subset W_3\subset W_6)$ has dimension $17$ ($6$ dimensions 
 for the choice of $W_3$ in $U_5$, transverse to $U_2$; then two extra dimensions for the choice of $W_1$ in $W_3$; finally nine extra dimensions for the choice of $W_6$ containing $W_3$). In this case $\cA\cap\cB$ has dimension $45$, so the rank of $\cP$
 is $17$ and we are safe. 
  \item[(1,3)] Let $T_1:=W_6\cap U_2$ and  $W_3:=W_6\cap U_5$.
     Here the parameter space has dimension $16$ and the dimension of $\cA\cap\cB$ jumps to $49$, so $\cP$ has rank $13$ on this stratum. So we need to dig a little bit more. Let  $W_8:=W_6+U_5$.
     In this case we check that 
     $$\cA\cap \cB=W_1\wedge U_5\wedge V_9+T_1\wedge W_6\wedge V_9+U_2\wedge (\wedge^2 W_6)+W_3\wedge U_5\wedge W_6\subset $$
     $$ \subset \cC := ((W_1+T_1)\wedge W_8\wedge V_9+\wedge^3 W_8)\cap \cA\subset\cA.$$
    The dimension of $\cC$ is $56$, so the rank of $\cA/\cC$ is $6$. 
    But now $\cC$ and $\cA/\cC$  only depend on $W_1$, $T_1$ and $W_8$. The parameter space for $(W_1,T_1,W_8)$ has dimension eight (one dimension for $T_1$ in $U_2$; four dimensions for $W_1$ in $U_5$; and three extra dimensions for $W_8$, that contains $U_5$). So the zero locus of the section induced by $v$ has dimension at most $2$ inside this stratum. 
     \item[(2,3)] Here $W_3:=W_6\cap U_5$ has dimension $3$ and $W_6\supset U_2$.
     In this case we check that 
     $$\cA\cap \cB=W_1\wedge U_5\wedge V_9+U_2\wedge W_6\wedge V_9\subset 
     \cC := W_1\wedge U_5\wedge V_9+U_2\wedge (\wedge^2V_9)\subset\cA.$$
    The bundle $\cA/\cC$ has rank four and depends only on $W_1$ for which we have four parameters; so we are safe. 
     \item[(1,4)] Let $T_1:=W_6\cap U_2$ and  $W_4:=W_6\cap U_5$. The parameter space for these flags has dimension $13$ (one dimension for $T_1$ in $U_2$; four dimensions for $W_1$ in $U_5$; then two dimensions for $W_4$ in $U_5$ and containing $T_1+W_1$; finally six extra dimensions for $W_6$ containing $W_4$). We check that 
          $$\cA\cap \cB=W_1\wedge U_5\wedge V_9+T_1\wedge W_6\wedge V_9+U_2\wedge (\wedge^2W_6)+
          W_4\wedge U_5\wedge W_6$$
          has dimension $50$. So $\cP$ has rank $12$ on this stratum, and since $\dim(S)-\rank_S(\cP)=1<3$, we are safe. 
            \item[(2,5)] This means that $U_5\subset W_6$. But then $\cA\cap\cB=\wedge^2U_5\wedge V_9$. 
            A dimension count shows that the general $v$ does not belong to any such space. \qed
\end{itemize}

Now, let us go back to the proof of the proposition. The only possibility for $[W_1]$ to be a general point of the hypersurface $\PP(U_5)\cap \cS$ is that $\dim(W_6\cap U_2)=2$ and $\dim(W_6\cap U_5)=4$. The parameter space for such pairs $[W_1\subset W_6]$ has dimension $11$ while the quotient bundle $\cP$ has rank $8$ on this stratum. So the set of points $[W_1]\in \PP(V_5)\cap \cS$ such that $\dim(U_6\cap V_2)=2$ and $\dim(U_6\cap V_5)=4$ has expected dimension $3$, i.e. it is a hypersurface inside $\PP(U_5)$. 

With the help of \cite{Macaulay2} we can construct the parameter space for the stratum $(2,4)$ and the quotient bundle  $\cP$, and check that the degree of $\cO_{\PP(U_5)}(1)$ on the zero locus of a general  section of $\cP$ inside the parameter space is equal to three. So this zero locus projects to a cubic hypersurface  in $\PP(U_5)$. Set theoretically  this cubic 
coincides with  $\PP(V_5)\cap \cS$. Since $\cS$ is a sextic, this means that $\cS$ cuts 
$\PP (U_5)$ along a double  cubic hypersurface. The result follows.
\end{proof}

In order to ensure we really get $\PP^4$'s in a ruling of $\SU_C(3)$, 
there remains to check that our family of linear spaces covers the whole moduli space. To be precise, we expect to get codimension one families of the two rulings. Indeed recall that 
the base of the rulings is five dimensional, exactly as $\overline{\Sing}(\cS)$. 
But the latter 
is actually, by its very definition, a $\PP^1$-fibration over its image $\pi_{2,5}(\overline{\Sing(\cS)})$ through the projection $\pi_{2,5}:Fl(1,2,5,V_9) \to Fl(2,5,V_9)$, whose 
birational image in $G(5,V_9)$ is the basis of our family. 
A similar phenomenon will happen for 
Hecke lines, as we will see in the next section. We have therefore a four
dimensional family of $\PP^4$'s and we want to show that these cover $\PP(V_9)$. 

\begin{prop}
There are $18$ $\PP^4$'s of the family $\pi_{2,5}(\overline{\Sing(\cS)})$ passing through the general point of $\PP(V_9)$. In particular, the $\PP^4$'s parametrized by $\pi_{2,5}(\overline{\Sing(\cS)})$ cover $\PP(V_9)$.
\end{prop}

\proof
This is a degree computation that we carried out with \cite{Macaulay2}. In order to do so, construct $\pi_{2,5}(\overline{\Sing(\cS)})$ inside $Fl(2,5,V_9)$ as the zero locus of a general section of the quotient of $\wedge^3 V_9$ by $\cU_2\wedge (\wedge^2 V_9) + \wedge^2 \cU_5 \wedge V_9$. Then consider the projective bundle $\PP(\cU_5)$ over $\pi_{2,5}(\overline{\Sing(\cS)})$ and its relative tautological bundle $\cO_{\PP(\cU_5)}(-1)$. We get $$\int_{\PP(\cU_5)}c_1(\cO_{\PP(\cU_5)}(1))^{8}=18$$
and the result follows.\qed

\medskip
In fact we get the two rulings of $\SU_C(3)$, since by Remark  \ref{remark_two_rulings}
they are exchanged by the covering involution of $\SU_C(3)$. Indeed the two copies of 
$\PP^4$ constituting $\theta^{-1}(\PP(U_5))$, being exchanged by the involution, must 
belong to different rulings. 

\section{The odd moduli space as an ODL}

\subsection{Orbital degeneracy loci in $G(3,V_9)$}
Over the Grassmannian, the Borel-Weil theorem provides still another identification 
$$\wedge^3V_9\simeq H^0(G(3,V_9),\wedge^3\cQ),$$
where $\cQ$ now denotes the rank six tautological quotient bundle. 
This allows to consider 
the skew-symmetric cubic tensor $v$ as a family of skew-symmetric cubic tensors in six
variables. The $GL(V_6)$-orbit closures in $\wedge^3V_6$ are easy to describe: by increasing 
dimension (and forgetting the extremal ones) we have 
\begin{itemize}
\item the cone $Z_{10}$ over the Grassmannian $G(3,6)$,
\item the variety $Z_5$ of partially decomposable tensors, 
\item the tangent variety $Z_1$ to the Grassmannian, which is a quartic hypersurface.
\end{itemize}
Correspondingly, we get inside $G(3,V_9)$
\begin{itemize}
\item a hypersurface $D_{Z_1}(v)$ of the Grassmannian, which is a singular quadratic section,
\item its singular locus  $D_{Z_5}(v)$, whose singular locus is 
\item  $D_{Z_{10}}(v)$, a smooth eightfold since $D_{Sing (Z_{10})}(v)=\emptyset$
 for dimensional reasons.
\end{itemize}

\begin{remark}
The hypersurface $D_{Z_1}(v)$ of $G(3,V_9)$ has obvious similarities with the Coble quadric discovered in \cite{coblequadric}. As the latter, it is a special quadratic 
section of the Grassmannian. Nevertheless its relationship with the moduli space is less direct. The Coble quadric can be characterized as the unique quadric section 
of the Grassmannian $G(2,V_8)$ which is singular exactly on a copy of the 
odd moduli space of rank two vector bundles on a (non hyperelliptic) genus three curve. Here, it is the singular locus of the quadric section, for which we are not aware of any modular interpretation, which is itself singular along a copy of the 
odd moduli space of rank three vector bundles on the genus two curve (indeed, we will show in Theorem \ref{thm_main2} that $D_{Z_{10}}(v)$ identifies with such a moduli space). The quadric section is presumably uniquely determined by this property, but we did not prove this.
\end{remark}


\subsection{A Fano eightfold}

By definition,  $D_{Z_{10}}(v)\subset G(3,V_9)$ is the projection from $Fl(3,6,V_9)$
of the locus 
\[
\tilde{D}_{Z_{10}}(v):=\{ (U_3\subset U_6)\in Fl(3,6,V_9),  \quad v\in U_3\wedge (\wedge^2V_9)
+ \wedge^3 U_6 \}. 
\]

\begin{prop}\label{Fano}
$\tilde{D}_{Z_{10}}(v)\simeq D_{Z_{10}}(v)$ is a smooth Fano eightfold of even index.
\end{prop}

\proof 
First observe that $\tilde{D}_{Z_{10}}(v)$ is defined as the zero locus 
of a general section 
of a globally generated vector bundle of rank $19$ on the flag manifold $Fl(3,6,V_9)$, namely $$\cE = \wedge^3V_9/(\cU_3\wedge (\wedge^2V_9)
+ \wedge^3 \cU_6).$$
We deduce that $\tilde{D}_{Z_{10}}(v)$ is smooth of dimension eight. Moreover its canonical bundle 
can be computed by the adjunction formula: since the canonical bundle of $Fl(3,6,V_9)$
is $\det(\cU_3)^6\otimes \det(\cU_6)^6$, we find that 
$$K_{\tilde{D}_{Z_{10}}(v)} = \det(\cU_3)^{-3}\otimes \det(\cU_6)^5_{|\tilde{D}_{Z_{10}}(v)}.$$
This does not seem to be negative, but observe that the quotient of  $\cU_3\wedge (\wedge^2V_9)+ \wedge^3 \cU_6$ by $\cU_3\wedge (\wedge^2 V_9)$ is just the line bundle $\wedge^3(\cU_6/\cU_3)$. 
Over $\tilde{D}_{Z_{10}}(v)$, our $v$ defines a section of this line bundle, and this section 
is nowhere vanishing for $v$ general. Indeed  $U_3\wedge (\wedge^2V_9)$ has codimension $30$ in $\wedge^3V_9$, and this is much bigger than the dimension of $G(3,V_9)$. We conclude that on $\tilde{D}_{Z_{10}}(v)$, the line bundle $\wedge^3(\cU_6/\cU_3)$ is trivial, which means that  $\det(\cU_3)\simeq \det(\cU_6)$. As a consequence, we can rewrite the canonical bundle as 
$$K_{\tilde{D}_{Z_{10}}(v)} = \det(\cU_3)\otimes \det(\cU_6)_{|\tilde{D}_{Z_{10}}(v)}\simeq 
\det(\cU_6)^2_{|\tilde{D}_{Z_{10}}(v)}.$$
Since $\det(\cU_3)\otimes \det(\cU_6)$ is dual to the minimal very ample line bundle 
on the flag manifold $Fl(3,6,V_9)$, the first identity shows that 
$\tilde{D}_{Z_{10}}(v)$ is 
Fano. The second identity implies  it has even index.  

We now  prove that  $\tilde{D}_{Z_{10}}(v)$ and $D_{Z_{10}}(v)$ are isomorphic. 
Let $\pi$ be the projection from $Fl(3,6,V_9)$ to $G(3,V_9)$. We denote 
in the same way its restriction from $\tilde{D}_{Z_{10}}(v)$ to $D_{Z_{10}}(v)$. Since 
$U_3\wedge (\wedge^2V_9)$ has codimension one inside $U_3\wedge (\wedge^2V_9)
+ \wedge^3U_6$, 
the fact that $[U_3]\in D_{Z_{10}}(v)$ has two pre-images in $\tilde{D}_{Z_{10}}(v)$ would imply that 
$v$ actually belongs to $U_3\wedge (\wedge^2V_9)$, and we already noticed this cannot be the case
for $v$ general. So $\pi :\tilde{D}_{Z_{10}}(v)\ra D_{Z_{10}}(v)$ is a bijection. 
There remains to check it is everywhere immersive.
For this we use the normal exact sequence 
$$0\lra T\tilde{D}_{Z_{10}}(v)\lra TFl(3,6,V_9)_{|\tilde{D}_{Z_{10}}(v)}\lra \cE_{|\tilde{D}_{Z_{10}}(v)}\lra 0.$$
Since the flag manifold is homogeneous, its tangent space at some point $(U_3\subset U_6)$ 
can be seen as a quotient of $\fsl(V_9)$, by the space of morphisms that preserve the 
flag. The projection to $G(3,V_9)$ induces a surjective 
map between tangent spaces, whose kernel, the vertical tangent space, is just the tangent space  $\Hom(U_3, U_6/U_3)$ to the fiber $G(3,U_6)$. This vertical tangent space should 
be seen as the quotient of the subspace of $\fsl(V_9)$ consisting of morphisms that 
preserve $U_3$. So the fact that $\pi$ is immersive at $(U_3\subset U_6)$ amounts to 
the following claim: if $X\in\fsl(V_9)$ preserves $U_3$ and sends $v$ to 
 $U_3\wedge (\wedge^2V_9)+ \wedge^3 U_6$, then it preserves the full flag 
 $(U_3\subset U_6)$. 
 
 In order to check this, note that we can decompose $v=v'+e_4\wedge e_5\wedge e_6$,
 where $v'$ belongs to $U_3\wedge (\wedge^2V_9)$ (which does not contain $v$, as we 
 already stressed), and $e_4, e_5, e_6$ are a basis of $U_6$ modulo $U_3$. Since $X$
 preserves $U_3$, $X(v')$ belongs to $U_3\wedge (\wedge^2V_9)$. On the other hand, 
 since $X(e_4\wedge e_5\wedge e_6)=X(e_4)\wedge e_5\wedge e_6+e_4\wedge X(e_5)\wedge e_6+e_4\wedge e_5\wedge X(e_6)$, it is clear that it belongs to $U_3\wedge (\wedge^2V_9)+ \wedge^3 U_6$ if and only if $X$ preserves $U_6$, and the proof is complete. \qed 
 
\subsection{A family of planes in the moduli space}

We want to prove that the projective planes in $\PP(V_9)$ parametrized by 
$D_{Z_{10}}(v)$ lift to planes in the moduli space $\SU_C(3)$. As we already noticed for the $\PP^4$'s of the rulings,
this happens if the sextic $\cS$ restricted to such a plane restricts to a 
double cubic.

\begin{prop}
\label{prop_hecke_in_coble2}
Let $[U_3]\in D_{Z_{10}}(v)$, then $\PP(U_3)\cap \cS $ is a double cubic, and 
$\theta^{-1}(\PP(U_3))$ is the union of two projective planes meeting along a cubic curve.
\end{prop}

\begin{proof}
Since $[U_3]\in D_{Z_{10}}(v)$, there exists $U_6\supset U_3$ such that $v\in U_3\wedge (\wedge^2 V_9) + \wedge^3 U_6$. We are looking for flags $(W_1\subset W_6)$ such that $W_1\subset U_3$ and $v\in W_1\wedge (\wedge^2 V_9) + (\wedge^2 W_6) \wedge V_9$. These flags correspond to points $[W_1]\in \PP(U_3)\cap \cS$. For a generic element $[W_1]$ in this intersection, 
we want to compute the values of $(\dim(W_6\cap U_3),\dim(W_6\cap U_6))$. 
We will prove:

\begin{lemma} 
$W_6\supset U_3$ and  $\dim(W_6\cap U_6)= 5.$
\end{lemma}

\proof For convenience let us denote $\cA= U_3\wedge (\wedge^2 V_9) + \wedge^3 U_6$ and $\cB=
W_1\wedge (\wedge^2 V_9) + (\wedge^2 W_6) \wedge V_9$. According to the relative position
of the flags $(U_3\subset U_6)$ and $(W_1\subset W_6)$, the intersection of $\cA$ and 
$\cB$ will have a specific dimension, and will be constrained by the fact that it has to
contain the general tensor $v$. Let us discuss these constraints more precisely, for 
each possible values of $(\dim(W_6\cap U_3),\dim(W_6\cap U_6))$; this couple of integers 
can a priori be either $(3,6)$, $(3,5)$, $(3,4)$, $(3,3)$, $(2,5)$, $(2,4)$, $(2,3)$, $(1,4)$, or $(1,3)$. We will discuss these cases roughly by increasing complexity, and exclude all of them except $(3,5)$. 



\begin{itemize}
\item[(3,3)] In this case 
 $\cA\cap\cB = W_1\wedge (\wedge^2 V_9) + U_3\wedge W_6\wedge V_9$, and we would deduce that 
 $v\in\cA\cap \cB\subset U_3\wedge (\wedge^2 V_9)$. But this is a 
contradiction, since we have already observed that the general tensor $v$ is contained 
in no such space.
\item[(2,3)] Here we compute that 
$$\cA\cap\cB = W_1\wedge (\wedge^2 V_9) + (U_3\cap W_6)\wedge W_6\wedge V_9+U_3\wedge (\wedge^2W_6),$$
So again  $\cA\cap\cB\subset U_3\wedge (\wedge^2 V_9)$ 
and we get the same contradiction as before. 
\item[(3,4)] Same argument. 
\item[(1,4)] In this case the parameter space for the flag $(U_1\subset U_6)$ has dimension $14$. Moreover
$$\cA\cap\cB = W_1\wedge (\wedge^2 V_9) +U_3\wedge (\wedge^2W_6)+\wedge^3U_6$$
has dimension $49$. So $v$ defines a general section of the vector bundle $\cA/\cA\cap\cB$, 
of rank $16$ over the $14$-dimensional parameter space, and will therefore vanish nowhere in general.
\item[(2,5)] Here the parameter space for the flag $(U_1\subset U_6)$ has dimension $9$, while 
$$\cA\cap\cB = W_1\wedge (\wedge^2 V_9) +(W_6\cap U_3)\wedge W_6\wedge V_9+U_3\wedge (\wedge^2W_6)+\wedge^3U_6$$
has dimension $53$. So the vector bundle $\cA/\cA\cap\cB$ has rank $12$ and we conclude as 
in the previous case.
  \item[(3,6)] In this case $U_6=V_6$ and the parameter space is a projective plane. Moreover $$\cA\cap\cB=W_1\wedge (\wedge^2 V_9) + U_3\wedge U_6\wedge V_9+\wedge^3U_6$$
   has dimension $59$. So the vector bundle $\cA/\cA\cap\cB$ has rank $6$ and we conclude as 
before.
\item[(1,3)] This is the biggest stratum, of dimension $17$. Here 
$$\cA\cap\cB = W_1\wedge (\wedge^2 V_9) +U_3\wedge (\wedge^2W_6)+\wedge^2(U_6\cap W_6)\wedge U_6$$
has dimension $49$, so the quotient $\cA/\cA\cap\cB$ has  rank $16$
which is smaller than the dimension of the parameter space. So we need to modify 
the argument a little bit. For this we observe that 
$$\cA\cap\cB\subset \cC:=W_1\wedge (\wedge^2 V_9) + \wedge^2(W_6+U_3)\wedge U_3+\wedge^3U_6,$$
which has dimension $54$ but only depends on $W_1$ and $W_8:=W_6+U_3$. And now the parameter
space for the flag $(W_1\subset U_3\subset W_8)$ is only $7$-dimensional, while the bundle 
$\cA/\cC$ has rank $11$. So the previous argument applies. 
\item[(2,4)] This case is somewhat similar to the previous case. Indeed the parameter space for the flag $(U_1\subset U_6)$ has dimension $13$, and we compute that 
$$\cA\cap\cB=W_1\wedge (\wedge^2 V_9) + (W_6\cap U_3)\wedge W_6\wedge V_9+ U_3\wedge (\wedge^2W_6)+\wedge^2(W_6\cap U_6)\wedge U_6$$
has dimension $53$, so the quotient $\cA/\cA\cap\cB$ has  rank $12$
which is smaller than the dimension of the parameter space. But we observe that $$\cA\cap\cB\subset \cD:=(W_6\cap U_3)\wedge (\wedge^2 V_9) +  \wedge^2(W_6+U_6)\wedge U_3+\wedge^3U_6,$$
which has dimension $60$ but only depends on $W_2:=W_6\cap U_3$ and $W_8:=W_6+U_6$. And now the parameter
space for the flag $(W_2\subset U_3\subset U_6\subset W_8)$ is only $4$-dimensional, while the bundle 
$\cA/\cD$ has rank $5$. So the previous argument applies. 
\end{itemize}

So we are only left with the case $(\dim(U_6\cap V_3),\dim(U_6\cap V_6))=(3,5)$,
which concludes the proof of the Lemma. \qed 

\medskip 
In order to conclude the proof of the Proposition, we proceed as follows. 
The parameter space for the flags $W_1\subset U_3\subset W_6$ with 
$\dim(U_6\cap W_6)\ge 5$ has dimension equal to $7$. Over this space we have 
$$\cA\cap\cB=W_1\wedge (\wedge^2 V_9) + U_3\wedge W_6\wedge V_9+ \wedge^3 U_6,$$
which has constant dimension $59$, even when $W_6$ specializes to $U_6$. So
$\cA/\cA\cap\cB$ defines a  rank six vector bundle $\cW$ on the parameter space.

Then with the help of \cite{Macaulay2} we construct the parameter space and the vector 
bundle $\cW$, and we check that the degree of $\cO_{\PP(V_3)}(1)$ on the zero locus of a section of $\cW$ inside the parameter space is equal to three. So this zero locus projects to a cubic curve $E$ in $\PP(V_3)$, and set-theoretically this 
cubic coincides with $\PP(V_3)\cap \cS$. But $\cS$ is a sextic, so it has to cut $E$ 
doubly in $\PP(V_3)$.  The result follows.
\end{proof}

So we get a family of planes in $\SU_C(3)$. Let us denote by $\cF$ the family of lines contained in these planes. We need  to check that the  lines from $\cF$  cover the 
whole moduli space. This follows from the next statement:

\begin{prop}\label{Dcovers}
The family of planes parametrized by $D_{Z_{10}}(v)$ covers $\PP(V_9)$.
\end{prop}

\begin{proof}
It is sufficient to check that the eighth Chern class of the bundle $(\cO_{\PP(V_9)}(1))^{\oplus 8}$ over the projective bundle $p:\PP(\cU_3) \to D_{Z_{10}}(v)$ is different from zero. 
This can be done with the help of \cite{Macaulay2}. Let $\cL$ be the pullback to $\PP(\cU_3)$ of the restriction  of the bundle $\cO_{G(3,V_9)}(1)$ to $D_{Z_{10}}(v)$. By constructing $D_{Z_{10}}(v)$ as a zero locus inside $Fl(3,6,V_9)$, and $\PP(\cU_3)$ as a projective bundle over it, we checked that
$$\int_{\PP(\cU_3)}c_1(\cL)^2 c_1(\cO_{\PP(\cU_3)}(1))^{8}=18,$$
which implies the claim. Similarly, in order to check this non-vanishing, we could have used the fact that $\int_{\PP(\cU_3)}c_2(\cM) c_1(\cO_{\PP(\cU_3)}(1))^{8}=6\ne 0,$ where $\cM$ is defined as $\cM:=(p^*\cU_3/\cO_{\PP(\cU_3)}(-1))^\vee$, which can again be checked with \cite{Macaulay2}.
\end{proof}

\begin{remark}
The 18 appearing in the formula here above has a nice interpretation. Recall that $\PP(\cU_3)$ has dimension 10, hence - given a general point $p\in\PP(V_9)$ - there is a surface inside $D_{Z_{10}}(v)$ that parametrizes planes through $p$. Moreover each point $p$ defines two points $[E_1]$, $[E_2]=\tau^*[E_1]$ in $\SU_C(3)$. It turns out that the surface is a disjoint union of two $\PP^2$'s embedded via their anticanonical class since the degree $18$ is twice the degree of each $\PP^2$ inside $D_{Z_{10}}(v)\subset G(3,V_9)$ in the Plücker embedding. In view of Theorem \ref{thm_main2}, these two planes correspond to the two Hecke planes inside $\SU_C(3,\cO_C(c))$ defined as extensions of $E_1$ by $\cO_c$ and $E_2$ by $\cO_c$.

\end{remark}

We will also need the following technical result. 

\begin{lemma}\label{pasSing}
The general line in $\cF$ does not meet $\Sing(\cS)$.
\end{lemma}

\proof It is enough to prove that the general plane in $D_{Z_{10}}(v)$ meets  $\Sing(\cS)$
at most in a finite set. In order to prove this, we proceed as follows. A plane in $D_{Z_{10}}(v)$ corresponds to a flag $(U_3\subset U_6)$, a point in $\Sing(\cS)$
(or rather $\overline{\Sing}(\cS)$) to a flag $(U_1\subset U_2\subset U_5)$. The point 
is contained in the plane if $U_1\subset U_3$ and we need that $v\in\cC\cap \cD$ 
where 
$$\cC=U_3\wedge (\wedge^2V_9)+\wedge^3U_6, \qquad \cD = U_2\wedge (\wedge^2V_9)+
(\wedge^2U_5)\wedge V_9.$$
We will stratify the parameter space by the relative position of the two flags
$(U_1\subset U_3\subset U_6)$ and $(U_1\subset U_2\subset U_5)$. There are $23$ 
possibilities. We can eliminate $9$ of them right away by observing that if $U_2\cap U_3=U_2\cap U_6$, and $U_5\cap U_3$ has codimension at most one in $U_5\cap U_6$, then 
$\cC\cap \cD\subset U_3\wedge (\wedge^2V_9)$, which cannot happen in general. 

For the $14$ remaining strata $S$, we will compute the dimension of $S$ and the 
rank of $\cC\cap \cD$ on $S$. If $\dim S+\rank_S(\cC\cap \cD)\le 92$, then the set of flags
compatible with the general $v$ will have dimension at most $92-84=8$, and will project
to $D_{Z_{10}}(v)$ with finite (possibly empty) general fibers.

The data we computed are summarized in the following table, where the intersection
dimensions in the first four columns specify each stratum $S$. 
\begin{small}
$$\begin{array}{|c|c|c|c||c||c|}
\hline 
  U_2\cap U_3 & U_2\cap U_6 & U_5\cap U_3 & U_5\cap U_6 & \dim S &\rank_S(\cC\cap \cD) \\
 \hline 
  \hline 
 1 &1 &1 &3 & 46& 46 \\
  \hline 
 1&2 &1 &2 & 45& 46\\
  \hline 
 1&2 &1 &3 & 44&46 \\
  \hline 
 1& 1& 1& 4& 42 &46 \\
  \hline 
 1&2 & 1&4 & 41& 46\\
  \hline 
 1&2 &2 &3 & 42 &50 \\
  \hline 
 1& 1& 2& 4& 41& 50\\
  \hline 
 1& 2& 2& 4& 40&50 \\
  \hline 
 1& 2& 2& 5& 36& 50 \\
  \hline 
 2& 2& 2& 4& 38& 53 \\
  \hline 
 1& 2& 3& 4& 37&53 \\
  \hline 
 1&2 &3 &5 & 34& 53 \\
  \hline 
 2& 2& 2& 5& 33& 53\\
  \hline 
 2& 2& 3& 5& 32 &59 \\
 \hline
 \end{array}$$
\end{small}

\medskip
The sum of the two numbers in the righmost columns does never exceed $92$, and 
this implies our claim.  \qed

\begin{prop}\label{DHecke}
The family of lines in the planes parametrized by $D_{Z_{10}}(v)$ is a codimension 
one family of Hecke lines.
\end{prop}

\proof The family $\cF$ of lines covers $\SU_C(3)$ by the previous statement, 
so it must consist of Hecke lines  or lines contained in the rulings. Let $\overline{\cF}$
denote its image in $G(2,V_9)$, which is the image of a projective plane bundle over $D_{Z_{10}}(v)$. Since $D_{Z_{10}}(v)$ is Fano, it is rationally
connected, so $\overline{\cF}$ is rationally connected as well. If it was contained
in $\overline{\cF}_R$, we could consider its preimage inside $\cF_{R_1}$ and then
$\hat{\cF}_{R_1}$. By Lemmas \ref{unique-ruling} and \ref{pasSing}, the preimage 
$\hat{\cF}$  would still be 
rationally connected, and would therefore be contained in a fiber of the fibration of 
$\hat{\cF}_{R_1}$ to the abelian surface $\Pic^1(C)$. But these fibers have codimension 
two, while $\hat{\cF}$ is a divisor: indeed $\overline{\cF}$ is a projective plane bundle over $D_{Z_{10}}(v)$, so it has dimension $10$,  hence codimension one in $\cF_H$. 
This is a contradiction.\qed


\subsection{The odd moduli space} We finally prove our main result.
\begin{theorem}
\label{thm_main2}
For generic $v\in \wedge^3 V_9$, the orbital degeneracy locus $D_{Z_{10}}(v)$ is isomorphic to the moduli space $\SU_C(3,\cO_C(c))$ of semistable rank three vector bundles on $C$ with fixed determinant $\cO_C(c)$, for a certain point $c\in C$.
\end{theorem}
\begin{proof}
By Proposition \ref{prop_hecke_in_coble2}, each point of $D_{Z_{10}}(v)$ defines a pair of planes 
in $\SU_C(3)$, exchanged by the cover involution. By Proposition \ref{DHecke} these planes are Hecke planes, that cover the whole moduli space by Proposition \ref{Dcovers}. 
According to Proposition \ref{recap}, there are only two such covering families 
of Hecke planes, and they are exchanged by the cover involution. This implies that 
each point of $D_{Z_{10}}(v)$ defines a plane in each of the two families $\cP_H$ and $\cP_{H^*}$. We can therefore lift $D_{Z_{10}}(v)$, at least birationally, to a hypersurface in either $\cP_H$ and $\cP_{H^*}$. But recall from Proposition \ref{recap}
that these two families fiber over $C$. Since $D_{Z_{10}}(v)$
is Fano, hence rationally connected, its pre-image in $\cP_H$ is also rationally 
connected, hence contained in a fiber of the projection to $C$, over some point $c$.
Recall that this fiber is a birational image of $\SU_C(3,\cO_C(c)))$, and therefore
the latter maps birationally to $D_{Z_{10}}(v)$. Since  $\SU_C(3,\cO_C(c)))$ is smooth and
its Picard group is cyclic, this birational morphism has to be an isomorphism. 
\end{proof}

We know that the moduli space $\SU_C(3,\cO_C(c))$ has index two, so comparing 
with Proposition \ref{Fano} we see that its canonical bundle must be the restriction
of $\det(\cU_3)^2$ on the Grassmannian, or also $\det(\cU_3)\otimes \det(\cU_6)$  
on the flag manifold $Fl(3,6,V_9)$. We deduce:

\begin{coro}
$\tilde{D}_{Z_{10}}(v)$ is projectively equivalent to $\SU_C(3)$ in its anticanonical
embedding. 
\end{coro}

\subsection{Resolving the structure sheaf} 
An immediate consequence of our approach is that we can provide a
resolution of the structure sheaf of $\SU_C(3,\cO_C(c))$ inside $G(3,V_9)$. 
Indeed, by construction $D_{Z_{10}}(v)$ is modeled on the cone $Z_{10}$ over
$G(3,V_6)\subset \PP(\wedge^3V_6)$. The minimal resolution of its ideal 
$I_{Z_{10}}$ has been computed in 
\cite{pw86}. Letting $S=\CC[\wedge^3V_6]$, this minimal resolution is reproduced in Figure \ref{fig_res_idZ}.

\begin{figure}[t]
\small
$$\xymatrix@-2.5ex{
  0 & \\ I_{Z_{10}}\ar[u] &  \\  S_{21111}V_6^\vee \otimes S(-2) \ar[u] & \\
  (S_{321111}V_6^\vee\oplus S_{22221}V_6^\vee)\otimes S(-3)\ar[u] & \\ 
 S_{332211}V_6^\vee\otimes S(-4)\ar[u]  &(S_{522222}V_6^\vee\oplus S_{33333}V_6^\vee)\otimes S(-5)\ar[lu]\\
 & (S_{533322}V_6^\vee\oplus S_{444222}V_6^\vee\oplus S_{443331}V_6^\vee) \otimes S(-6) \ar[lu]\ar[u]\\
 & (S_{544332}V_6^\vee\oplus S_{544332}V_6^\vee)\otimes S(-7)\ar[u]  \\
 & (S_{554442}V_6^\vee\oplus S_{644433}V_6^\vee\oplus S_{555333}V_6^\vee) 
 \otimes S(-8)\ar[u] \\
 S_{665544}V_6^\vee\otimes S(-10) \ar[ur]
& (S_{555552}V_6^\vee\oplus S_{744444}V_6^\vee) \otimes S(-9)\ar[u] \\
  (S_{666654}V_6^\vee\oplus S_{765555}V_6^\vee)\otimes S(-11)\ar[u]\ar[ur] & \\
 S_{766665}V_6^\vee\otimes S(-12) \ar[u]&\\
 \det(V_6^\vee)^7\otimes S(-14) \ar[u]&\\
0\ar[u]& 
}$$
\normalsize
\caption{Minimal resolution of $I_{Z_{10}}$}
\label{fig_res_idZ}
\end{figure}
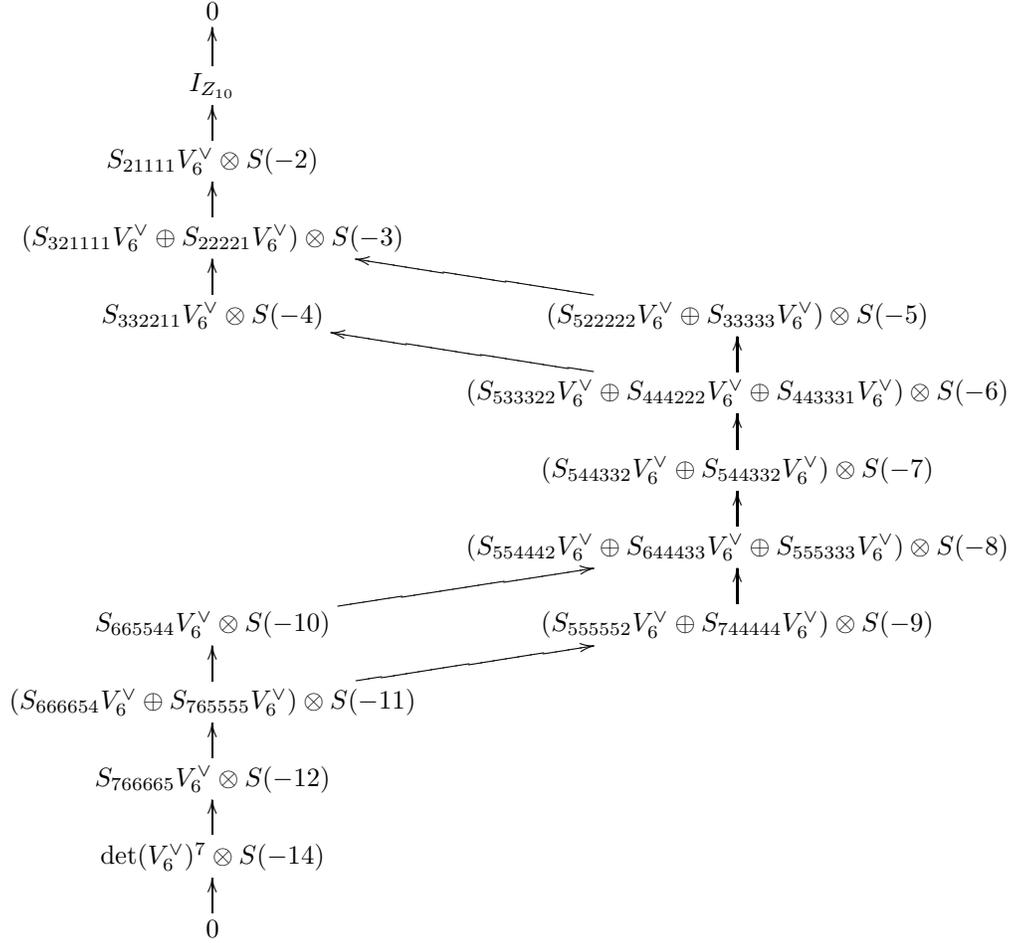

Since  this is an equivariant resolution, we can relativize it to the total space of 
$\wedge^3Q$ over $G(3,V_9)$. This yields the resolution of the ideal sheaf of $D_{Z_{10}}(v)$ in Figure \ref{fig_res_idD}.
\begin{figure}[h]
\small
$$\xymatrix@-2.5ex{
  0&  \\  I_{D_{Z_{10}}(v)}\ar[u] &  \\  S_{21111}\cQ^\vee  \ar[u] & \\
  S_{21}\cQ^\vee(-1)\oplus S_{22221}\cQ^\vee(1)\ar[u] & \\ 
 S_{2211}\cQ^\vee(-1)\ar[u]  &S_{3}\cQ^\vee(-2)\oplus S_{33333}\cQ^\vee\ar[lu]\\
 & S_{3111}\cQ^\vee(-2)\oplus S_{222}\cQ^\vee(-2)\oplus S_{33222}\cQ^\vee (-1) \ar[lu]\ar[u]\\
 & S_{32211}\cQ^\vee(-2)\oplus S_{32211}\cQ^\vee(-2)\ar[u]  \\
 & S_{33222}\cQ^\vee(-2)\oplus S_{3111}\cQ^\vee(-3)\oplus S_{222}\cQ^\vee(-3) 
 \ar[u] \\
 S_{2211}\cQ^\vee(-4) \ar[ur]
& S_{33333}\cQ^\vee(-2)\oplus S_{3}\cQ^\vee (-4)\ar[u] \\
  S_{22221}\cQ^\vee(-4)\oplus S_{21}\cQ^\vee(-5)\ar[u]\ar[ur] & \\
 S_{21111}\cQ^\vee(-5) \ar[u]&\\
 \cO(-7) \ar[u]&\\
0\ar[u]&
}$$
\normalsize
\caption{Minimal resolution of the ideal of $D_{Z_{10}}(v)$}
\label{fig_res_idD}
\end{figure}

\begin{remark}
Notice the remarkable fact that this corresponds  to a Cohen-Macaulay (length equal to the codimension) and even Gorenstein (last term of rank one, which implies that the resolution is self-dual) locally free resolution of $\cO_{D_{Z_{10}}(v)}$. In general, if $\cF_\bullet$ is a  Gorenstein $\cO_X$-locally free resolution of $\cO_Z$, with $Z\subset X$, then the canonical bundle $K_Z$ of $Z$ is given by the formula:
$$ K_Z = (K_X \otimes F_{-\codim_X(Z)}^\vee)|_{Z}.$$
In our situation, we deduce that $K_{D_{Z_{10}}(v)}$ is the restriction of $K_{G(3,V_9)}\otimes \cO(7)=\cO(-2)$, which confirms what we proved in Proposition \ref{Fano}.
\end{remark}


We can derive the following facts from the previous resolution. 

\begin{coro}
The Hilbert polynomial of $({D_{Z_{10}}(v)},\cO_{D_{Z_{10}}(v)}(1))$ is
$$\chi(\cO_{D_{Z_{10}}(v)}(m-1))=
\frac{477}{2240}\,m^{8}+\frac{63}{160}\,m^{6}+\frac{99}{320}\,m^{4}+
\frac{47}{560}\,m^{2}.$$
\end{coro}

In full generality, the Hilbert polynomial of the moduli space of vector bundles of fixed rank and determinant on a curve of given genus is given by the Verlinde formula, but in a form which is far from being obviously polynomial, see e.g.  \cite[Proposition 3.1]{beauville-verlinde}. 

\begin{prop}
${D_{Z_{10}}(v)}$ has the following properties.
\begin{enumerate}
\item It is not linearly normal, but $k$-normal for any $k\ge 2$.
\item It is contained in $810$ quadratic sections of the Grassmannian.
\item It is scheme theoretically cut-out by quadrics. 
\end{enumerate}\end{prop}

\proof 
The Hilbert polynomial  (recall that ${D_{Z_{10}}(v)}$ is Fano) gives 
$h^0(\cO_{D_{Z_{10}}(v)}(1))=85$, which is bigger by one than the dimension of 
 $\wedge^3V_9$. The difference comes from the term $S^3\cQ^\vee(-2)$ in the resolution 
 of  $I_{D_{Z_{10}}(v)}$: by Bott's theorem $h^3(S^3\cQ^\vee(-1))=1$. From the same 
 theorem we readily reduce that $h^1(I_{D_{Z_{10}}(v)}(k))=0$ for any $k\ge 2$, which proves the first assertion. The second one is just a computation. The last one follows from the fact that $S_{21111}\cQ^\vee(2)=S_{21111}\cQ$ is generated by global sections,
 hence also $I_{D_{Z_{10}}(v)}(2)$. \qed

 \medskip
 It would be nice to find an explanation of this failure of linear normality.

\subsection{Application to $K3$ surfaces of genus $19$} 
Consider again the embedding 
$$D_{Z_{10}}(v)\simeq \SU_C(3, \cO_C(c)) \subset G(3,V_9).$$ 
Since the moduli space
is a Fano eightfold of index two, we will obtain a surface $S$ with trivial canonical 
bundle
by considering the zero locus of two general sections of the dual tautological bundle (or equivalently, the intersection of $D_{Z_{10}}(v)$ with a general translate of $G(3,7)$ inside $G(3,V_9)$). 

\begin{prop}\label{k3}
The surface $S$ is a K3 surface of genus $19$. 
\end{prop}

\proof 
In order to simplify the notations we let $D:=D_{Z_{10}}(v)$ and $G:=G(3,V_9).$
We need to check that $S$ is connected and that $h^1(\cO_S)=0$. This a standard computation: the Koszul complex of the section  of $\cG=\cU^\vee\oplus \cU^\vee$ that defines $S$ in $D$ resolves its structure sheaf as a $\cO_D$-module, and we will be able to conclude if we can check that 
$$h^i(D,\wedge^i\cG^\vee)=h^i(D,\wedge^{i-1}\cG^\vee)=0\quad \forall i>0.$$
Since the bundle $\cG$ on $D$ is the restriction of a bundle defined on the whole Grassmannian, we can check these vanishing by twisting the complex $\cJ^\bullet$ of vector bundles on the Grassmannian that resolves $\cO_D$, and verify that 
\begin{equation}\label{eq_cohom_K3}
    h^{i+j}(G,\wedge^i\cG^\vee\otimes \cJ^j)=h^{i+j}(D,\wedge^{i-1}\cG^\vee\otimes \cJ^j)=0
\quad \forall i>0, \forall j\ge 0.
\end{equation}
This is just an application of Bott's theorem on the Grassmannian. All bundles in play are homogeneous and even completely reducible: they come from certain semisimple representations of the parabolic subgroup $P$ defining $G(3,V_9)=\GL(9)/P$. The exterior powers and tensor products appearing in \eqref{eq_cohom_K3} can thus be computed with \cite{Lie} as exterior powers and tensor products of representations of the Levi factor of $P$, which is just $GL(3)\times GL(6)$. Then, it is possible to implement Bott's Theorem with a Python script to compute the cohomology of the corresponding bundles. In this way we were able to obtain all cohomology groups $H^{k}(G,\wedge^i\cG^\vee\otimes \cJ^j)$ for all $i\geq 0$, $j\geq 0$, $k\geq 0$. In fact they all vanish except $H^{0}(G,\wedge^0\cG^\vee\otimes \cJ^0)$ and $H^{18}(G,\wedge^6\cG^\vee\otimes \cJ^{10})$, which are both one dimensional. This implies that $h^0(\cO_S)=1$, $h^1(\cO_S)=0$ and $h^2(\cO_S)=1$, proving that $S$ is a $K3$ surface. The genus is $19$ because the degree of $\cO_S(1):=\cO_G(1)|_S$ is equal to $36$, as computed with \cite{Macaulay2}. \qed

\begin{remark} Similar computations lead to the following observations and questions:
\begin{enumerate}
    \item The Euler characteristic $\chi(S,\cU^\vee_{|S})=9$, suggesting that $\cU^\vee_{|S}$ has nine global sections,  while we would expect seven. What is the explanation? 
    \item Since the normal bundle of $S$ in $D=\SU_C(3,\cO_C(p)) $ is $N_{S/D}=
    \cU^\vee_{|S}\oplus \cU^\vee_{|S}$, this would imply that 
    $h^0(S,N_{S/D})=9+9=18$. Do the intersection of the moduli space
    with translates of $G(3,7)$ really form a family of K3 surfaces of dimension $18$?
    \item The Euler characteristic $\chi(S,\mathcal{E}nd_0(\cU^\vee_{|S}))=-2$. 
    If we could prove that $\cU^\vee_{|S}$ is stable, hence simple, we would deduce that his moduli space is a K3 surface. It is conjectured in \cite[Section 1.2]{km} that this moduli space should be a quartic K3 surface.
\end{enumerate}

\end{remark} 

\begin{remark} The construction of Hecke lines and planes in $\SU_C(3)$ shows that 
the odd moduli space  $\SU_C(3,\cO_C(p))$ has a different embedding inside 
$G(3,V_9)$ for each point $p\in C$. Equivalently, since this moduli space does not depend on $p$, it admits a rank three vector bundle $\cE_p$ for each $p$, generated by global sections. Considering a global section of $\cE_p\oplus\cE_q$, for $p$ and $q$ any two points in $C$, it follows from Proposition \ref{k3} that we will in general still get a  K3 surface. 
Mukai and Kanemitsu recently suggested that this construction 
should yield a locally complete family of K3 surfaces of genus $19$, a genus for which no such projective 
model was previously known \cite{km}. \end{remark}

\bibliography{moduli}

\bibliographystyle{amsalpha.v2}

\footnotesize

\bigskip 
Universit\'e Côte d’Azur, CNRS – Laboratoire J.-A. Dieudonné, Parc Valrose, F-06108 Nice Cedex 2, {\sc France}.

{\it Email address}: {\tt vladimiro.benedetti@univ-cotedazur.fr}

\smallskip 

Institut Montpelli\'erain Alexander Grothendieck, Université de Montpellier, 
 Place  Eugène Bataillon, 34095 Montpellier Cedex 5, {\sc France}.
 
{\it Email address}: {\tt michele.bolognesi@umontpellier.fr}

\smallskip 

Institut de Mathématiques de Bourgogne, Université de 
Bourgogne et Franche-Comté, 9 Avenue Alain Savary, 21078 Dijon Cedex, {\sc France}.

{\it Email address}: {\tt daniele.faenzi@u-bourgogne.fr}

\smallskip 
Institut de Math\'ematiques de Toulouse, 
Paul Sabatier University, 118 route de Narbonne, 31062 Toulouse
Cedex 9, {\sc France}.

{\it Email address}: {\tt manivel@math.cnrs.fr}
\end{document}